\documentclass{amsart}
\usepackage{amsmath}
\usepackage{amssymb}
\usepackage{amscd}
\usepackage{color}
\usepackage{enumerate}

\begin{document}

\title[Structures and lower bounds for binary covering arrays]
{Structures and lower bounds for  binary covering arrays}
\author[Soohak Choi]
{Soohak Choi}
\author[Hyun Kwang Kim]
{Hyun Kwang Kim}
\address{(Choi, Kim) Department of Mathematics, Pohang University of Science and Technology,
Pohang 790-784, Republic of Korea.} \email{misb@postech.ac.kr, hkkim@postech.ac.kr}
\author[Dong Yeol Oh]
{Dong Yeol Oh}
\address{(Oh) Division of Liberal Arts, College of Humanities and Sciences, Hanbat National University, Daejeon 305-719, Republic of Korea.} \email{dyoh@postech.ac.kr}

%
\thanks{The second author's work was supported by Basic Science Research Program through the National Research Foundation of Korea(NRF) funded by the Ministry of Education, Science and Technology (KRF 2008-314-C00007).}
\thanks{Key Words : Covering arrays, Erd\"{o}s-Ko-Rado theorem, Roux's bound.}


\newtheorem{theorem}{Theorem}[section]
\newtheorem{lemma}[theorem]{Lemma}
\newtheorem{definition}[theorem]{Definition}
\newtheorem{cor}[theorem]{Corollary}
\newtheorem{proposition}[theorem]{Proposition}
\newtheorem{example}[theorem]{Example}
\newtheorem{remark}[theorem]{Remark}

\def\LI{\langle}
\def\RI{\rangle}
\def\supp{\mbox{supp}}
\def\Aut{\mbox{Aut}}
\def\Iso{\mbox{Iso}}
\newtheorem{pppp}{Proof}
\renewcommand{\thepppp}{\hspace{-5pt}}
\newenvironment{pf}{\begin{pppp} \em}{\mbox{}\hfill\qed\end{pppp}}



\begin{abstract}
A $q$-ary $t$-covering array is an $m \times n$ matrix with entries from $\{0, 1, \ldots, q-1\}$ with the property that for any $t$ column positions, all $q^t$ possible vectors of length $t$ occur at least once. One wishes to minimize $m$ for given $t$ and $n$, or maximize $n$ for given $t$ and $m$. For $t = 2$ and $q = 2$, it is completely solved by R\'{e}nyi, Katona, and Kleitman and Spencer. They also show that maximal binary $2$-covering arrays are uniquely determined. Roux found the lower bound of $m$ for a general $t, n$, and $q$. In this article, we show that $m \times n$ binary $2$-covering arrays under some constraints on $m$ and $n$ come from the maximal covering arrays. We also improve the lower bound of Roux for $t = 3$ and $q = 2$, and show that some binary $3$ or $4$-covering arrays are uniquely determined.

\end{abstract}

\maketitle

\section{Introduction.}
Let $B_q=\{0,1,\ldots,q-1\}$ be a set with $q$ elements. An $m\times n$
matrix $C$ over $B_q$ is called a $t$-covering array (or a covering
array of size $m$, strength $t$, degree $n$, and order $q$) if for any $t$
columns of $C$, all $q^t$ possible $q$-ary vectors of length $t$ occur at least
once. Such an array will be denoted by $CA(m;t,n,q)$.

The problem
is to minimize $m$ for which a $CA(m;t,n,q)$ exists for given values of $q, n$, and $t$, or equivalently to maximize $n$ for which a $CA(m;t,n,q)$ exists for given values of $q, m$, and $t$. Such a minimal size $m$ and a maximal degree $n$ are denoted by $CAN(t,n,q)$ and  $\overline{CAN}(t,m,q)$, respectively.
%
%
For fixed $t$ and $n$, a $t$-covering array of degree $n$ with minimal size
$m=CAN(t,n,q)$ is called optimal.

The problem was completely solved only for the case $t = q = 2$ by R\`{e}nyi \cite{Re} (for $m$ even), and independently Katona \cite{Kat}, and Kleitman and Spencer \cite{Kle} (for all $m$): the answer is that for any $m$, the maximal degree of a binary $2$-covering array is
$$\overline{CAN(2,m,2)} = {m - 1\choose \lfloor \frac{m}{2} \rfloor - 1 }.$$
Such an array with maximal degree is called a maximal covering array. Moreover, Katona \cite{Kat} proved that maximal binary covering arrays of strength $2$ are uniquely determined up to equivalence.

For a higher strength $t \geq 3$ or a higher order $q \geq 3$, the problem becomes more difficult. For examples, when $t = 2$ and $q = 3$, $CAN(2,n,3)$ is known only for $n \leq 5$. (See \cite{Slo}.) For a general $t, n$, and $q$, Roux \cite{Rou} introduced two useful bounds of $CAN(t,n,q)$.

Covering arrays have wide applications in combinatorial sciences such as circuit testing, intersecting codes, data compression, and so on. See \cite{Bie, Boy, Bra, Car, Coh, Coh2, Coh3}, \cite{Gal}, \cite{Nao}, \cite{Slo, Ste, Ste2, Ton}.


%

We are interested in the structures of binary optimal $2$ or $3$-covering arrays and the lower bound of $CAN(t,n,q)$.

Let $C$ be an $m \times n$ $q$-ary $t$-covering array. It is easy to see that if we permute
the rows and columns of $C$ or invert the values of any column of $C$ by a permutation of $B_{q}$, then the resulting matrix is also a $t$-covering array. This motivates two covering arrays $C$ and $C'$ are equivalent if one can be
transformed into the other by a series of operations of the
following types:\\
(a) permutation of the rows;\\
(b) permutation of the columns;\\
(c) inversion of the values of any column by a permutation of $B_{q}$.

Johnson and Entringer \cite{Joh} showed that $CAN(n-2,n,2)= \lfloor \frac{2^{n}}{3} \rfloor$ and that the corresponding covering array is unique. Colbourn et al. \cite{Col} classified the number of inequivalent covering arrays for up to degree $10$, order $8$, and all possible strengths by a computer search.

The purpose of this article is to classify the structures of some optimal binary $2$-covering arrays, and to improve the lower bound of Roux on $CAN(t,n,q)$  when $t = 3, q = 2$.

In Section 3, we will show that when $n > {{{m-1}\choose{\lfloor \frac{m}{2}\rfloor-1}}+m-3{\lfloor \frac{m}{2}\rfloor}}$, binary optimal $2$-covering arrays of size $m$ and degree $n$ is obtained from the maximal $2$-covering of size $m$ by deleting some columns by using a combinatorial approach.

In Section 4, we will improve the lower bound of Roux on $CAN(3,n,2)$  when $n > {{{m-1}\choose{\lfloor \frac{m}{2}\rfloor-1}}+m-3{\lfloor \frac{m}{2}\rfloor}}$ and $m \geq 7$.

In Section 5,  we will show that $10 \times 5, 12 \times 11$ binary optimal $3$-covering and $24 \times 12$ binary optimal $4$-covering arrays are unique by using the results in Section 3. The results in Section 5, except Theorem \ref{Thm:48by12}, are already known in Colbourn et al. \cite{Col}; they found these result by a computer search.

\section{Preliminaries}

In this section, we will introduce some definitions and basic concepts which are needed in the sequel.

For $u=(u_1,u_2,\ldots,u_n) \in B_q^n$, the support supp$(u)$ and the weight $wt(u)$ of $u$ are defined to be
$$
\mbox{supp}(u)=\{i \bigm | u_i \neq 0\} \mbox{ and } wt(u)=|\mbox{supp}(u)|.
$$
Let $C$ be an $m \times n$ matrix over $B_q$. We denote $c^i$ and $r^j$ by $i$-th column and $j$-th row of $C$, respectively.

When $q=2$, we sometimes consider $u \in B_2^n$ as a subset of $[n]=\{1,\ldots,n\}$ by identifying a binary vector with its support.
The complement $\overline{u}$ of $u \in B_2^n$ is defined by
$$
\overline{u}_i=\left\{
            \begin{array}{ll}
              1, & \ \mbox{if} \ u_i=0; \\
              0, & \ \mbox{if} \ u_i=1.
            \end{array}
          \right.
$$
For a matrix $C$ over $B_2$, the complement $\overline{C}$ of a matrix $C$ is defined by
$$
\overline{c}_{ij}=\left\{
            \begin{array}{ll}
              1, & \ \mbox{if} \ c_{ij}=0; \\
              0, & \ \mbox{if} \ c_{ij}=1.
            \end{array}
          \right.
$$
An $m \times n$ matrix $C$ over $B_2$ is a $t$-covering array if for any $t$
columns $c^{i_1},c^{i_2},\ldots,c^{i_t}$ of $C$, $\bigcap_{j=1}^{t} X_j \neq \emptyset$, where $X_j$ is either $\mbox{supp}(c^{i_j})$ or $\mbox{supp}(\overline{c^{i_j}})$.

Kapralov \cite{Kap} introduced a residual matrix which is useful to study the structures of covering arrays.
\begin{definition}
Let $C$ be a matrix over $B_q$. Let $c^{i_{1}},c^{i_{2}},\ldots,c^{i_{k}}$ be different columns of a matrix $C$. The residual matrix
$Res(C;c^{i_{1}}=v_1,c^{i_{2}}=v_2,\ldots,c^{i_{k}}=v_k)$ is the submatrix of $C$ obtained by the following way$:$ take all the rows in which $C$ has
value $v_j$ in the column $c^{i_{j}}$ for $j=1,2,\ldots,k$ and delete the
columns $c^{i_{1}},c^{i_{2}},\ldots,c^{i_{k}}$ in the selected rows.
\end{definition}

From the definition of the residual matrix, the followings can be easily obtained.
\begin{proposition}{\label{Prop:residual}}
Let $C$ be a $t$-covering array over $B_q$. For $k < t$, the residual matrix
$Res(C;c^{i_{1}}=v_1,c^{i_{2}}=v_2,\ldots,c^{i_{k}}=v_k)$ of $C$ is a $(t-k)$-covering array over $B_q$.
\end{proposition}

\begin{proposition}{\label{Prop:weight}}
Let $C$ be an $m \times n$ $t$-covering array over $B_q$. Then for any $1 \leq i \leq n$, the weight $wt(c^i)$ of $i$-th column of $C$
satisfies
$$
(q-1)CAN(t-1,n-1,q) \leq wt(c) \leq m - CAN(t-1,n-1,q).
$$
\end{proposition}

For $u, v \in B_q^n$, the distance $d(u,v)$ of $u$ and $v$ is defined to be
$$
d(u,v)=|\{i \bigm | u_i \neq v_i\}|.
$$
For an $m \times n$ matrix $C$ over $B_{q}$, the set $R(C)$  is defined to be
$$
R(C)=\{v_{i} \bigm | 1 \leq i \leq m \},
$$
where
$$
v_{i}=(e_{i}^{1},e_{i}^{2},\ldots,e_{i}^{n}) \mbox{ and } e_{i}^{j} = |\{k \bigm | d(r^{i},r^{k}) = j\}|.
$$
We call $R(C)$ row distance structure.

By using the definition of equivalence of covering arrays and row distance structure,
we can get a necessary condition for two covering arrays to be equivalent, which is useful to determine whether two covering arrays are equivalent
.

\begin{proposition}{\label{Prop:rowdistance}}
If $C$ and $C'$ are equivalent $t$-covering arrays over $B_q$, then $R(C) = R(C')$.
\end{proposition}

Now we will introduce a typical example of binary $2$-covering arrays.

\begin{definition}
The standard maximal binary $2$-covering array $C$ of size $m$ is an $m \times {{m-1}\choose{{\lfloor \frac{m}{2} \rfloor}-1}}$ matrix such that
\begin{enumerate}
\item the first row of $C$ is all $1$ row;
\item the columns of the remaining matrix is the family of all vectors of $({\lfloor \frac{m}{2} \rfloor -1})$ $1$'s and ${\lceil \frac{m}{2} \rceil}$ $0$'s.
\end{enumerate}
\end{definition}

From the definition of the standard maximal binary $2$-covering array, we can get the trivial lower bound of the degree of binary $2$-covering arrays of size $m$.
\begin{proposition}{\label{Prop:standard}}
For $m \geq 4$,
$$
\overline{CAN}(2,m,2) \geq {{m-1}\choose{{\lfloor \frac{m}{2}
\rfloor}-1}}.
$$
\end{proposition}

We close this section by introducing a famous Hall's theorem. Let $G=(V,I)$ be a graph with a vertex set $V$ and an edge set $I$.
For a subset $S$ of $V$, let $\Gamma(S)$ be the set of neighborhoods of $S$ in $G$, i.e. the set of vertices adjacent to any element of $S$.

\begin{theorem}$($E. W. Hall 1935$)$\label{Thm:Hall}\\
Suppose we have a bipartite graph $G$ with two vertex sets $V_1$ and $V_2$. Suppose that
$$
|\Gamma(S)| \geq |S| \qquad \textrm{for every }S \subset V_{1}.
$$
Then $G$ contains a complete matching.
\end{theorem}

\section{Structures of some optimal binary $2$-covering arrays}
In this section, we investigate structures of binary $2$-covering arrays of size $m$ and degree $n$ when $n>{{{m-1}\choose{\lfloor \frac{m}{2}\rfloor-1}}+m-3{\lfloor \frac{m}{2}\rfloor}}$.
Throughout this section, a $2$-covering array means a binary $2$-covering array.\\
Let $C$ be a $2$-covering array of size $m$ and degree $n$ and $c^i$ be the $i$-th column of $C$.
By interchanging $c^i$ with its complement, we may assume that $wt(c^{i}) \leq \lfloor \frac{m}{2} \rfloor$ for all $1 \leq i \leq n$. So every $2$-covering array $C$ is equivalent to a $2$-covering array $C'$
of the same size and degree with $wt(c'^i) \leq \lfloor \frac{m}{2} \rfloor$ for all $i \in [n]$.

\begin{lemma}\label{Lem:Hall}
Let $C$ be a $2$-covering array of size $m$ and degree $n$ with $wt(c^i) \leq \lfloor \frac{m}{2} \rfloor$ for all $1 \leq i \leq n$. Put $s=min_{1 \leq i \leq n}{wt(c^{i})}$. For any integer $s'$ satisfying $s < s' \leq {\lfloor \frac{m}{2} \rfloor}$, there is a $2$-covering
array $C'$ of size $m$ and degree $n$ with $s' \leq {wt(c'^{i})} \leq \lfloor \frac{m}{2} \rfloor$ such that $supp(c^{i}) \subseteq supp(c'^{i})$ for all $i \in [n]$.
\end{lemma}
\begin{proof}
Let $C_i$ be the set of columns of $C$ whose weight is $i$. Let $W_j$ be the set of binary vectors of length $m$ whose weight is $j$. We consider the bipartite graph $G$ with vertex sets $C_s$ and $W_{s+1}$ and edge set $E=\{cc'|c \in C_s, c' \in W_{s+1}\mbox{ and supp}(c)\subseteq \mbox{supp}(c')\}$. For $c \in C_s$, there are $(m-s)$ vectors in $W_{s+1}$ whose support contains $\mbox{supp}(c)$ and for $u\in W_{s+1}$, there are at most $(s+1)$ columns whose support is contained in $\mbox{supp}(u)$. For every $S \subseteq C_s$,
$$
|\Gamma(S)| \geq \frac{m-s}{s+1} |S| \geq \frac{\lceil \frac{m}{2} \rceil + 1}{\lfloor \frac{m}{2} \rfloor} |S| > |S|.
$$
By applying Theorem \ref{Thm:Hall} to $G$, $G$ contains a complete matching $f$ from $C_s$ to $W_{s+1}$. Let $C'$ be the $m \times n$ matrix obtained from $C$ by following way: if a column $c$ of $C$ does not belong to $C_{s}$, we keep it; otherwise replace $c$ with $f(c)$. Then, $C'$ is a matrix such that for each $i$, $s+1 \leq {wt(c'^{i})} \leq \lfloor \frac{m}{2} \rfloor$ and $\mbox{supp}(c^i) \subseteq \mbox{supp}(c'^i)$.

We claim that $C'$ is also a $2$-covering array. Let $X_i$ be either $\mbox{supp}(c'^i)$ or $\mbox{supp}(\overline{c'^i})$. It is enough to show that $X_i \cap X_j \neq \emptyset$ for $i \neq j$.

Since $C$ is a $2$-covering array and $\mbox{supp}(c^i) \subseteq \mbox{supp}(c'^i)$ for any $i$, $\emptyset \neq \mbox{supp}(c^i) \cap \mbox{supp}(c^j) \subseteq \mbox{supp}(c'^i) \cap \mbox{supp}(c'^j)$.

If $|\mbox{supp}(c^j)| > s$, then $\mbox{supp}(c'^j) = \mbox{supp}(c^j)$. Since $C$ is a $2$-covering array, $\emptyset \neq \mbox{supp}(c^i) \cap \mbox{supp}(\overline{c^j}) \subseteq \mbox{supp}(c'^i) \cap \mbox{supp}(\overline{c'^j})$. 
If $|\mbox{supp}(c^i)| > s$ and $|\mbox{supp}(c^j)| = s$, then $|\mbox{supp}(c^i) \cap \mbox{supp}(\overline{c^j})| \geq 2$ since $\mbox{supp}(c^i) \nsupseteq \mbox{supp}(c^j)$.
Since $\mbox{supp}(c^j) \subset \mbox{supp}(c'^j)$ and $1 + |\mbox{supp}(c^j)| = |\mbox{supp}(c'^j)|$, $|\mbox{supp}(c^i) \cap \mbox{supp}(\overline{c'^j})| = |\mbox{supp}(c^i) \cap \mbox{supp}(\overline{c^j})| - 1 \geq 1$. Since $c'^i = c^i$, $\mbox{supp}(c'^i) \cap \mbox{supp}(\overline{c'^j}) \neq \emptyset$.
If $|\mbox{supp}(c^i)| = |\mbox{supp}(c^j)| = s$, then $|\mbox{supp}(c'^i)| = |\mbox{supp}(c'^j)| = s+1$. Since $f$ is a complete matching, $\mbox{supp}(c'^i) \neq \mbox{supp}(c'^j)$. Since $|\mbox{supp}(c'^i) \cap \mbox{supp}(\overline{c'^j})| = |\mbox{supp}(c'^i)| - |\mbox{supp}(c'^i) \cap \mbox{supp}(c'^j)| \geq (s+1) - s = 1$, $\mbox{supp}(c'^i) \cap \mbox{supp}(\overline{c'^j}) \neq \emptyset$. Thus regardless of the weights of $c^i$ and $c^j$, we have $\mbox{supp}(c'^i) \cap \mbox{supp}(\overline{c'^j}) \neq \emptyset$.
By symmetry, we also have $\mbox{supp}(\overline{c'^i}) \cap \mbox{supp}(c'^j) \neq \emptyset$.

Since $|\mbox{supp}(c'^i)| \leq \lfloor \frac{m}{2} \rfloor$ and $|\mbox{supp}(c'^j)| \leq \lfloor \frac{m}{2} \rfloor$, $|\mbox{supp}(c'^i) \cup \mbox{supp}(c'^j)| = |\mbox{supp}(c'^i)| + |\mbox{supp}(c'^j)| - |\mbox{supp}(c'^i) \cap \mbox{supp}(c'^j)| = \lfloor \frac{m}{2} \rfloor + \lfloor \frac{m}{2} \rfloor - 1 \leq m-1$. Hence $|\mbox{supp}(\overline{c'^i}) \cap \mbox{supp}(\overline{c'^j})| = m - |\mbox{supp}(c'^i) \cup \mbox{supp}(c'^j)| \geq 1$,  thus $\mbox{supp}(\overline{c'^i}) \cap \mbox{supp}(\overline{c'^j}) \neq \emptyset$.
\end{proof}

Similarly, we can obtain the followings.

\begin{cor}\label{Cor:Hall01}
Let $C$ be a $2$-covering array of size $m$ and degree $n$ with $wt(c^i) \leq \lfloor \frac{m}{2} \rfloor$ for all $i \in [n]$ and $wt(c^j)< \lfloor \frac{m}{2} \rfloor$. Then there is a $2$-covering array $C'$ of size $m$ and degree $n$ with $wt(c'^{j})=\lfloor \frac{m}{2} \rfloor - 1$ and $wt(c'^{i})=\lfloor \frac{m}{2} \rfloor$ for all $i \in [n]$ and $i\neq j$ such that $supp(c^{i}) \subseteq supp(c'^{i})$ for all $i \in [n]$.
\end{cor}
\begin{cor}\label{Cor:Hall02}
Let $C$ be a $2$-covering array of size $m$ and degree $n$ with $wt(c^i) \leq \lfloor \frac{m}{2} \rfloor$ for all $1 \leq i \leq n$. Then there is a $2$-covering array $C'$ of size $m$ and degree $n$ with $wt(c'^{i})=\lfloor \frac{m}{2} \rfloor$ for all $1 \leq i \leq n$ such that $supp(c^{i}) \subseteq supp(c'^{i})$ for all $i \in [n]$.
\end{cor}

We introduce well known theorem called the Erd\"{o}s-Ko-Rado theorem without proof. See \cite{Lin} for a proof.

\begin{theorem}$($Erd\"{o}s-Ko-Rado 1938$)$\label{Thm:Erd}\\
If $m \geq 2r$, and $\mathcal{F}$ is a family of distinct subsets of $[m]$ such that each subset is of size $r$ 
and each pair of subsets intersects, then the maximum number of sets that can be in $\mathcal{F}$ is given by the binomial coefficient
$$
{m-1}\choose{r-1}.
$$
\end{theorem}

\begin{proposition}{\label{Prop:Hall}}
For $m \geq 4$,
$$
\overline{CAN}(2,m,2) \leq {{m-1}\choose{{\lfloor \frac{m}{2}
\rfloor}-1}}.
$$
\end{proposition}
\begin{proof}
Let $C$ be a $2$-covering array of size $m$ and degree $n$. From Corollary \ref{Cor:Hall02}, we can assume $wt(c^{i})=\lfloor \frac{m}{2} \rfloor$ for all $i \in [n]$. It follows from Theorem \ref{Thm:Erd}.
\end{proof}

From Propositions \ref{Prop:standard} and {\ref{Prop:Hall}}, we have

\begin{theorem}$($Katona 1973, Kleitman and Spencer 1973$)$\label{Thm:Sperner}\\
For $m \geq 4$,
$$
\overline{CAN}(2,m,2) = {{m-1}\choose{{\lfloor \frac{m}{2}
\rfloor}-1}}.
$$
\end{theorem}

Hilton and Milner \cite{Hil} gave a upper bound to the degree of $2$-covering arrays of size $m$ under some condition.

\begin{theorem}$($Hilton and Milner 1967$)$\label{Thm:Hilton}\\
Let $2 \leq k \leq {\lfloor \frac{m}{2} \rfloor}$ and $C$ be a $2$-covering array of size $m$ and degree $n$ with $wt(c^{i}) \leq k$ for all $i \in [n]$ and
${\bigcap_{1 \leq i \leq n}}{ \mbox{supp}(c^{i}) }={\emptyset}.$ Then
\begin{equation}\label{eq:hil1}
n \leq {1+{{m-1}\choose{k-1}}-{{m-k-1}\choose{k-1}}}.
\end{equation}
There is strict inequality in $(\ref{eq:hil1})$ if $wt(c^{i}) < k$ for some $i \in [n]$.
\end{theorem}

Put $k = \lfloor \frac{m}{2} \rfloor$ in Theorem \ref{Thm:Hilton}, we have
\begin{cor}\label{Cor:hil}
Let $C$ be a $2$-covering array of size $m$ and degree $n$ with $wt(c^i) \leq \lfloor \frac{m}{2} \rfloor$ for all $i \in [n]$ and ${\bigcap_{1 \leq i \leq n}}{ \mbox{supp}(c^{i}) }={\emptyset}.$ Then
\begin{equation}\label{eq:hil2}
n \leq \left\{
\begin{array}{ll}
{{m-1}\choose{ \lfloor \frac{m}{2} \rfloor - 1}}, & \textrm{if} \ \ m \ \textrm{is even} ;\\
{{m-1}\choose{\lfloor \frac{m}{2} \rfloor - 1}}- \lfloor \frac{m}{2} \rfloor + 1, & \textrm{if} \ \ m \ \textrm{is odd}.
\end{array} \right.
\end{equation}
There is strict inequality in $(\ref{eq:hil2})$ if $wt(c^{i}) < \lfloor \frac{m}{2} \rfloor$ for some $i \in [n]$.
\end{cor}

We now state the main result of this section.


\begin{theorem}{\label{Thm:equi}}
Let $C$ be a $2$-covering array of size $m$ and degree $n$. If $m \geq 4$ and $n>{{{m-1}\choose{\lfloor \frac{m}{2}\rfloor-1}}+m-3{\lfloor \frac{m}{2}\rfloor}}$, then $C$ is equivalent to
$C'$, where $C'$ is made from deleting columns
of standard maximal $2$-covering of size $m$. Moreover, for each column $c^i$ of $C$
\begin{equation}\label{eq:hil2}
wt(c^i) = \left\{
\begin{array}{ll}
{\lfloor \frac{m}{2}\rfloor} \ \ \ \ \ \ \ \ \ \ \ \ \ \ \ \ \ \ \textrm{if m is even },\\
{\lfloor \frac{m}{2}\rfloor} \ \mbox{or} \ {\lfloor \frac{m}{2}\rfloor} + 1 \ \ \ \textrm{if m is odd }.
\end{array} \right.
\end{equation}
\end{theorem}

\begin{proof}
By the definition of equivalence of covering arrays, $C$ is equivalent to a $2$-covering array $C_1$,
where $wt(c_1^i) \leq {\lfloor \frac{m}{2}\rfloor}$ for any column $c_1^i$ of $C_1$. Hence we may assume that $wt(c^{i}) \leq {\lfloor \frac{m}{2}\rfloor}$ for each column $c^i$ of $C$.\\
Let $m \geq 5$ be odd. Since $n>{{{m-1}\choose{{\lfloor \frac{m}{2}\rfloor}-1}}+m-3{\lfloor \frac{m}{2}\rfloor}}$, $\bigcap_{1 \leq i \leq n}{\mbox{supp}(c^{i})}\neq{\emptyset}$ by Corollary $\ref{Cor:hil}$. We may assume that $1 \in \bigcap_{1
\leq i \leq n}{\mbox{supp}(c^{i})}$. It is enough to show that $wt(c^{i})={\lfloor \frac{m}{2}\rfloor}$ for each column $c^i$ of $C$.
Suppose there is a column $c^i$ of $C$ such that $wt(c^{i})<{\lfloor \frac{m}{2}\rfloor}$. Without loss of generality, we can assume $wt(c^{n}) <{\lfloor \frac{m}{2}\rfloor} $.
By Corollary $\ref{Cor:Hall01}$, there is a $2$-covering array $C'$ such that $wt(c'^{i}) = {\lfloor \frac{m}{2}\rfloor}$ for $i \neq n$ and $wt(c'^{n}) = {\lfloor \frac{m}{2}\rfloor}-1$. Let $D$ be the $m \times (n-1)$ submatrix of $C'$ obtained from $C'$ by deleting the last
column. Since $\mbox{supp}(c'^{n}) \nsubseteq \mbox{supp}(c'^{i})$ for any $i \neq n$ and $D$ is a $2$-covering array whose first row is all 1's vector, the number of columns of $D$ is at most ${{{m-1}\choose{{\lfloor \frac{m}{2}\rfloor}-1}}-(m-{\lfloor \frac{m}{2}\rfloor}+1)} = {{{m-1}\choose{{\lfloor \frac{m}{2}\rfloor}-1}}+m-3{\lfloor \frac{m}{2}\rfloor}}-3$. However, the number of columns of $D$ is $n-1$ which is greater than or equal to ${{{m-1}\choose{{\lfloor \frac{m}{2}\rfloor}-1}}+m-3{\lfloor \frac{m}{2}\rfloor}}$. It is a contradiction.\\
Let $m \geq 4 $ be even. We claim that $wt(c^i) = {\lfloor \frac{m}{2}\rfloor}$ for each column $c^i$ of $C$. If there is a column $c$ of $C$ with $wt(c) \neq {\lfloor \frac{m}{2}\rfloor}$, then by the same argument as one in the odd case, we may assume that $wt(c^i) = {\lfloor \frac{m}{2}\rfloor}$ and $wt(c^n) < {\lfloor \frac{m}{2}\rfloor}$ for $i \neq n$. By Corollary $\ref{Cor:Hall01}$, there is a $2$-covering array $C'$ such that $wt(c'^{i}) = {\lfloor \frac{m}{2}\rfloor}$ for $i \neq n$, $wt(c'^{n}) = {\lfloor \frac{m}{2}\rfloor}-1$, and $\mbox{supp}(c^j) \subseteq \mbox{supp}(c'^j)$ for all $j$. By the definition of equivalence of covering arrays, we may also assume that $1 \in \mbox{supp}(c'^n)$.
Since $wt(c'^{i}) = {\lfloor \frac{m}{2}\rfloor}$ for $i \neq n$ and $m$ is even, $wt(\overline{c'^i}) = {\lfloor \frac{m}{2}\rfloor}$ for $i \neq n$. After inversions of suitable columns of $C'$, we can get an $m \times n$ $2$-covering array $C''$ such that $wt(c''^{i}) = {\lfloor \frac{m}{2}\rfloor}$ for $i \neq n$, $wt(c''^{n}) = {\lfloor \frac{m}{2}\rfloor}-1$, and $1 \in \bigcap_{1
\leq i \leq n}{\mbox{supp}(c''^{i})}$. By the same argument as one when $m$ is odd, we can also get a contradiction.
\end{proof}

\begin{cor}{\label{Cor:max2ca1}}
Every maximal $2$-covering array of size $m$ is equivalent to the
standard maximal $2$-covering array of size $m$. Thus, maximal $2$-covering arrays of size $m$ are unique.
\end{cor}
\begin{cor}{\label{Cor:max2ca2}}
If $m\geq6$ and $n={{m-1}\choose{\lfloor \frac{m}{2} \rfloor}-1}-1$,
then every $2$-covering array $C$ of size $m$ and degree $n$ is equivalent to a $2$-covering array $C'$ of size $m$ and degree $n$, where $C'$ is made from deleting a column of
the standard maximal binary $2$-covering array of size $m$. Thus, $2$-covering arrays of size $m \geq 6$ and degree $n={{m-1}\choose{\lfloor \frac{m}{2} \rfloor}-1}-1$ are unique.
\end{cor}

\begin{remark}
When $m$ is odd,  Corollary \ref{Cor:max2ca1} and \ref{Cor:max2ca2} are also shown in \cite{Oh}.
\end{remark}

Using Corollary \ref{Cor:max2ca1} and \ref{Cor:max2ca2}, and Proposition $\ref{Prop:rowdistance}$, we can classify the number of nonequivalent $2$-covering arrays satisfying $CAN(2,n,2)=6$.

$$
\begin{array}{ccccccccccccccc}
n & \vline & 6 & 7 & 8 & 9 & 10\\
\hline
CA(6;2,n,2) & \vline & 4 & 3 & 1 & 1 & 1
\end{array}
$$
$$
\textrm{Table } 1: \textrm{The number of covering arrays }CA(6;2,n,2).
$$

\section{Lower bounds of some binary $3$-covering arrays}

In this section, we will give a new lower bound of size $m$ for a binary $3$-covering array of degree $n$.
Roux \cite{Rou} gave two useful bounds of $CAN(t,n,q)$. We will improve the lower bound of $CAN(t,n,q)$ given by Roux when $t = 3$ and $q = 2$.

We introduce the Roux's bound without proof. See Theorem 6 in \cite{Rou} for a proof.

\begin{theorem}{\label{Thm:Roux}}
For any positive integers $t,n$ and $q$,
\begin{eqnarray*}
CAN(t+1,n+1,q) &\geq& q CAN(t,n,q),\nonumber\\
CAN(3,2n,2) &\leq& CAN(3,n,2) + CAN(2,n,2).
\end{eqnarray*}
\end{theorem}


To improve the lower bound $CAN(3,n,2)$, we need some lemmas.

\begin{lemma}{\label{Lem:dist}}
Let $C$ be a $2m \times (n+1)$ binary $3$-covering array. If ${{{m-1}\choose{\lfloor \frac{m}{2}
\rfloor-1}}+m-3\lfloor \frac{m}{2}
\rfloor} < n \leq
{{m-1}\choose{\lfloor \frac{m}{2}
\rfloor-1}}$ and $m \geq 5$, then $wt(c^i) = m$ for each column $c^i$ of $C$. Moreover, $d(c^i,c^j)=2{\lfloor \frac{m}{2}
\rfloor}$ or $2{\lceil \frac{m}{2} \rceil}$ for any distinct columns $c^i$ and $c^j$ of $C$.
\end{lemma}

\begin{proof}
Let $C$ be a $2m \times (n+1)$ binary $3$-covering array where $m$ and $n$ are satisfying the assumption.
We claim that $wt(c^i) = m$ for any column $c^i$ of $C$.\\
Suppose that there is a column, say $c^1$, of $C$ whose weight is not equal to $m$.
By the definition of equivalence, we may assume that $wt(c^1) = k < m$. Then $Res(C; c^1 = 1)$ is a $k \times n$ binary $2$-covering array.
Since $\overline{CAN}(2,k,2) = {{k-1} \choose {\lfloor \frac{k}{2} \rfloor} - 1}$ and $k < m$, we have
$$n \leq {{k-1} \choose {\lfloor \frac{k}{2} \rfloor - 1}} \leq {{m-2} \choose {\lfloor \frac{m-1}{2} \rfloor - 1}}.$$
After a direct computation, it can be easily shown that
$${{m-2} \choose {\lfloor \frac{m-1}{2} \rfloor - 1}} \leq {{{m-1}\choose{\lfloor \frac{m}{2}
\rfloor-1}}+m-3\lfloor \frac{m}{2}
\rfloor}  \ \mbox{if} \ \ m \geq 5.$$
It is a contradiction to the condition of $m$ and $n$. Therefore, $wt(c^i) = m$ for any column $c^i$ of $C$. For each $i$, $Res(C; c^i = 1)$ is an $m \times n$ binary $2$-covering array with ${{{m-1}\choose{\lfloor \frac{m}{2}
\rfloor-1}}+m-3\lfloor \frac{m}{2} \rfloor} < n \leq {{m-1}\choose{\lfloor \frac{m}{2} \rfloor-1}}$. By Theorem \ref{Thm:equi}, each column of $Res(C; c^i = 1)$ has
weight $\lfloor \frac{m}{2} \rfloor$ or $\lceil \frac{m}{2} \rceil$. Since $wt(c^i) = m$ for each column $c^i$ of $C$, $d(c^i, c^j) = 2 \lfloor \frac{m}{2} \rfloor$ or $2 \lceil \frac{m}{2} \rceil$
for any distinct columns $c^i$ and $c^j$ of $C$.
\end{proof}

After a direct computation, we have the following lemma.

\begin{lemma}{\label{Lem:ineq1}}
The followings are hold.
\begin{enumerate}[$(a)$]
\item If $l \geq 4$, then ${{2l-1}\choose{l-1}} > 5l$,
\item If $l \geq 5$, then ${{2l}\choose{l-1}} \geq 4l^2$.
\end{enumerate}
\end{lemma}


\begin{lemma}{\label{Lem:ineq4}}
If $l \geq 4$ and ${{2l}\choose{l-1}}-l+2 \leq n \leq
{{2l}\choose{l-1}}$, then ${\frac{n}{2} - \frac{\sqrt{n}}{2}}
> {{2l-1}\choose{l-2}}$.
\end{lemma}

\begin{proof}
When $l=4$ and $54 \leq n \leq 56$, it holds.
Let $f(x)={\frac{1}{2}x - \frac{1}{2}\sqrt{x}}$. Since
$f'(x)={\frac{1}{2} - \frac{1}{4\sqrt{x}}} > 0$ for any $x >
\frac{1}{4}$, it is enough to show that $f(x) >
{{2l-1}\choose{l-2}}$
when $l \geq 5$ and $x={{2l}\choose{l-1}}-l+2$. By Lemma ${\ref{Lem:ineq1}}$ $(b)$,
\begin{eqnarray}
f(x)-{{2l-1}\choose{l-2}}&=&{\frac{x}{2} - \frac{\sqrt{x}}{2}} - {{2l-1}\choose{l-2}}\nonumber\\
&=&{\frac{1}{2}}\left\{\left ({{2l}\choose{l-1}}-l+2\right ) -
\sqrt{{{2l}\choose{l-1}}-l+2}\right\} - {{2l-1}\choose{l-2}}\nonumber\\
&>&\left (\frac{1}{2}{{2l}\choose{l-1}} - {{2l-1}\choose{l-2}} -
\frac{l-2}{2} \right ) - \frac{1}{2}\sqrt{{{2l}\choose{l-1}}}\nonumber\\
&=&\left (\frac{1}{2l}{{2l}\choose{l-1}} - \frac{1}{2}\sqrt{{{2l}\choose{l-1}}}\right ) -
\frac{l-2}{2} \nonumber\\
&=& \frac{1}{2} \sqrt{{{2l}\choose{l-1}}} \left (\frac{1}{l} \sqrt{{{2l}\choose{l-1}}} -1 \right ) - \frac{l-2}{2}\nonumber\\
&\geq& l - \frac{l-2}{2} = \frac{l+2}{2} > 0 .
\end{eqnarray}
\end{proof}

Nurmela \cite{Nur} found a $15 \times 12$ binary $3$-covering array by tabu search and Colbourn et al. \cite{Col} proved $CAN(3,12,2) = 15$ by a computer search. Hence we deduce that $CAN(3,15,2) \geq 15$ and $CAN(3,16,2) \geq 15$. We will give a combinatorial proof of $CAN(3,15,2) \geq 15$ and $CAN(3,16,2) \geq 15$.

\begin{lemma}{\label{Lem:15by16}}
The covering array number of strength $3$ and degree $16$ over $B_2$ is greater than or equal to $15$, i.e. $CAN(3,16,2) \geq 15$.
\end{lemma}
\begin{proof}
It is enough to show that there is no $14 \times 16$ binary $3$-covering array. Let $C$ be a $14 \times 16$ binary $3$-covering array.
It follows from \ref{Lem:dist} that $wt(c^i) = 7$ for each column $c^i$ of $C$. Without loss of generality, we may assume that the first row of $C$ is all $1$'s vector.
Then $Res(C;c^1=1)$ and $Res(C;c^1=0)$ are $7 \times 15$ binary $2$-covering arrays. Hence $Res(C;c^1=1)$ is the standard maximal binary $2$-covering
array of size $7$ by Theorem ${\ref{Thm:equi}}$. So the weight of any column of $Res(C;c^1=1)$ is $3$. Since the weight of any column of $C$ is $7$, the weight of any column of $Res(C;c^1=0)$ is $4$. Hence the weight of any column of $\overline{Res(C;c^1=0)}$ is $3$.
By Theorem ${\ref{Thm:equi}}$, $\overline{Res(C;c^1=0)}$ is also the standard maximal binary $2$-covering array of size $7$. Note that the first rows of $Res(C;c^1=1)$ and $\overline{Res(C;c^1=0)}$ are all $1$'s vectors. For each row of $Res(C;c^1=1)$ and $\overline{Res(C;c^1=0)}$ except the first rows of each array, there are five $1$'s and ten $0$'s. Hence $\sum_{2 \leq i < j \leq 16}{d(c^i,c^j)} = 2\cdot6\cdot5\cdot10=600$. However, since $d(c^i,c^j)=6$ or $8$ for any $i \neq j$ by Lemma
${\ref{Lem:dist}}$, we have
$$
630=6\cdot{{15}\choose{2}}\leq \sum_{2\leq i < j \leq
16}{d(c^i,c^j)} \leq 8\cdot {{15}\choose{2}}=840.
$$
It's a contradiction.
\end{proof}

By the same argument as one in Lemma \ref{Lem:15by16}, we have

\begin{lemma}{\label{Lem:15by15}}
The covering array number of strength $3$ and degree $15$ over $B_2$ is greater than or equal to $15$, i.e. $CAN(3,15,2) \geq 15$.
\end{lemma}

We now state the main results of this section, which improve the lower bound of Roux.

\begin{theorem}{\label{Thm:Roux_all}}
If $m \geq 7$ is odd and
${{{m-1}\choose{\lfloor \frac{m}{2} \rfloor-1}}+m-3\lfloor \frac{m}{2} \rfloor} < n \leq {{m-1}\choose{\lfloor \frac{m}{2} \rfloor-1}}$, then\\
$CAN(3,n+1,2) \geq 2CAN(2,n,2)+1$.
\end{theorem}

\begin{proof}
When $m=7$, it is done by Lemma
${\ref{Lem:15by16}}$ and Lemma ${\ref{Lem:15by15}}$. We
assume that $m \geq 9$ is odd. We note that there is an $m \times n$ binary $2$-covering array by the conditions of $m$ and $n$. Suppose that $C$ is an $2m \times (n + 1)$ binary $3$-covering array. By Proposition \ref{Prop:weight}, $wt(c^i) = m$ for $1 \leq i \leq n$. Hence $Res(C;c^1=1)$ and $Res(C;c^1=0)$ both are $m \times n$ $2$-covering arrays. By
Theorem ${\ref{Thm:equi}}$, we can also assume that $Res(C;c^1=1)$ and
$\overline{Res(C;c^1=0)}$ are made from deleting columns of the standard
binary $2$-covering array of size $m$. For each row, except the first row, of the standard binary $2$-covering array of size $m$, there are
${m-2}\choose{\lfloor \frac{m}{2} \rfloor-2}$ $1$'s and ${m-2}\choose{\lfloor \frac{m}{2} \rfloor-1}$ $0$'s.
Since $Res(C;c^1=1)$ and $\overline{Res(C;c^1=0)}$ are obtained from the standard binary $2$-covering array by deleting some columns,  there are at most
${m-2}\choose{\lfloor \frac{m}{2} \rfloor-2}$ $1$'s in each row of $Res(C;c^1=1)$ and $\overline{Res(C;c^1=0)}$ except the first rows of each array.
Hence
$$
\sum_{2\leq i < j \leq n+1}{d(c^i,c^j)} \leq
{2(m-1){{m-2}\choose{\lfloor \frac{m}{2} \rfloor-2}}\left (n-{{m-2}\choose{\lfloor \frac{m}{2} \rfloor-2}}\right )}
$$
By Lemma ${\ref{Lem:dist}}$,
$$
(m-1){{n}\choose{2}} \leq \sum_{2\leq i < j \leq n+1}{d(c^i,c^j)}
\leq (m+1){{n}\choose{2}}.
$$
By Lemma ${\ref{Lem:ineq4}}$,
\begin{eqnarray}
\sum_{2\leq i < j \leq n+1}{d(c^i,c^j)} &\leq&
{2(m-1){{m-2}\choose{\lfloor \frac{m}{2} \rfloor-2}}\left (n-{{m-2}\choose{\lfloor \frac{m}{2} \rfloor-2}}\right )}\nonumber\\
&=&2(m-1)\left\{\left(\frac{n}{2}\right)^2 - \left(\frac{n}{2}-{{m-2}\choose{\lfloor \frac{m}{2} \rfloor-2}}\right)^2\right\}\nonumber\\
&<&2(m-1)\left ( \left(\frac{n}{2}\right)^2-\left(\frac{\sqrt{n}}{2}\right)^2\right )
=(m-1){{n}\choose{2}}\nonumber\\
&\leq& \sum_{2\leq i < j \leq n+1}{d(c^i,c^j)}.\nonumber
\end{eqnarray}
It's a contradiction. Thus, $CAN(3,n+1,2) \geq 2CAN(2,n,2)+1$.
\end{proof}

\begin{theorem}{\label{Thm:Roux_even}}
If $m\geq 8$ is even and ${{m-1}\choose{\frac{m}{2}-1}}-\frac{m}{2} < n \leq
{{m-1}\choose{\frac{m}{2}-1}}$, then\\
$CAN(3,n+1,2)\geq 2 CAN(2,n,2)+2$.
\end{theorem}

\begin{proof}
It is enough to show that there is no $(2m+1)\times(n+1)$ binary $3$-covering
array. Let $C$ be a $(2m+1)\times(n+1)$ binary $3$-covering array. From Proposition ${\ref{Prop:weight}}$, $wt(c^i)=m$ or $m+1$ for
$1 \leq i \leq n + 1$. By taking complement of the columns with weight $m+1$, we may assume that
$wt(c^i)=m$ for any $i$. By Lemma ${\ref{Lem:dist}}$,
$d(c^i,c^j)=m$ for any pair $i,j$. Let $B$ be the
$(2m+1)\times(n+1)$ matrix obtained from replacing $0$'s by $-1$'s and $A$ be the $(n+1) \times (n+1)$ matrix ${B^T}B$. Then, $A=2m I + J$, where $I$ and $J$ are
$(n+1)\times(n+1)$ identity and all $1$'s matrix, respectively. The rank of $A$ is $(n+1)$. By Lemma ${\ref{Lem:ineq1}}$,
\begin{eqnarray*}
n + 1 = rank(A) \leq rank(B) \leq 2m + 1 < n + 1.
\end{eqnarray*}
It is a contradiction. Therefore, $CAN(3,n+1,2)\geq 2 CAN(2,n,2)+2$.
\end{proof}

By Theorems ${\ref{Thm:Roux}}$, ${\ref{Thm:Roux_all}}$, and ${\ref{Thm:Roux_even}}$, we have

\begin{cor}
If $m\geq 7$, $t\geq 3$ and
${{m-1}\choose{\lfloor \frac{m}{2} \rfloor-1}}+m-3\lfloor \frac{m}{2} \rfloor < n \leq {{m-1}\choose{\lfloor \frac{m}{2} \rfloor-1}}$, then
$$
CAN(t,n+t-2,2)\geq\left\{
\begin{array}{ll}
2^{t-3}(2m+1), & \textrm{if $m$ is odd}\\
2^{t-2}(m+1), & \textrm{if $m$ is even}.
\end{array} \right.
$$
\end{cor}

\section{Uniqueness of some optimal binary covering arrays}

In this section, we will show that for given $n$ and small $t$ ($t = 3, 4$), some binary optimal $t$-covering arrays of degree $n$ are unique. For a large $t = n - 2$, Johnson and Entringer \cite{Joh} constructed an infinite family of optimal binary $t$-covering arrays, and proved that such optimal covering arrays are unique. We will briefly introduce the result of Johnson and Entringer.

Let $Q_n$ be the graph whose vertices are the binary $n$-tuples
$v=(v_1,\ldots,v_n)$, two of which are adjacent if and only if they
differ in exactly one coordinate. For
$u=(u_1,\ldots,u_n)$ and $v=(v_1,\ldots,v_n)$, define
$w := u + v$ by $w=(w_1,\ldots,w_n)$, where $w_i \equiv u_i + v_i (mod \emph{ }2)$
for $1 \leq i \leq n$. We set $\lvert v \rvert = \Sigma_{i=1}^{n}{v_i}$. For $\emptyset \neq S \subseteq V(Q_n), c \in V(Q_n)$, define
$S+c$ by $S+c=\{s+c | s\in S\}$. The subgraph of $Q_n$ induced
by $S$ is denoted by $<S>$. And let $C_4$ be a $4$-cycle. The following is proved by Johnson and Entringer \cite{Joh}.

\begin{theorem}{\label{Thm:4cycle}}
Let $V_{n}^{j}=\{v \in V(Q_n) \bigm | \lvert v \rvert \equiv j (mod \
3)\}$ and set $S_n = V_{n}^{r_n} \cup V_{n}^{r_n-1}$, where $r_n$
is chosen from $\{0,1,2\}$ so that $n \equiv 2r_n$ or $2r_n-1 (mod \
6)$. Then for $n\leq 1 $,
\begin{enumerate}[a)]

\item If $S\subseteq V(Q_n)$ and $\lvert S \rvert > \lceil 2^{n+1}/3
\rceil$, then $<S>$ contains a $C_4$.

\item For all $c \in V(Q_n)$, $\lvert S_n+c \rvert = \lceil
2^{n+1}/3 \rceil$, and $<S_{n}+c>$ contains no $C_4$.

\item If $S \subseteq V(Q_n)$, $\lvert S \rvert = \lceil 2^{n+1}/3
\rceil$, and $<S>$ contains no $C_4$, then $S=S_n+c$ for some $c\in V(Q_n)$.

\end{enumerate}
\end{theorem}

A $t$-covering array of degree $n$ can be thought as a subgraph $G$ of
$n$-cube $Q_n$ such that every $(n-t)$-subcube contains a vertex of
$G$. Hence the following is an immediate consequence of Theorem $\ref{Thm:4cycle}$.

\begin{cor}
For $n \geq 4$, $CAN(n-2,n,2)=\lfloor 2^{n}/3 \rfloor$, and every
$\lfloor 2^{n}/3 \rfloor \times n$ covering array of strength
$(n-2)$ is equivalent to the matrix whose rows from the set
$V_{n}^{r_n+1}$ in Theorem $\ref{Thm:4cycle}$.
\end{cor}

Now we will show that $10 \times 5, 12 \times 11$ binary $3$-covering, $24 \times 12$ binary
$4$-covering arrays are unique. From Theorem \ref{Thm:Sperner}, it is easy to show that $CAN(2,4,2) = 5$. After a simple computation, we can easily get
\begin{lemma}{\label{Lem:5by4}}
Every $5 \times 4$ binary $2$-covering array is equivalent to
$$
\small{\left(
\begin{array}{cccc} 0 & 0 & 0 & 0\\
0 & 1 & 1 & 1\\
1 & 0 & 1 & 1\\
1 & 1 & 0 & 1\\
1 & 1 & 1 & 0 \end{array} \right)}
$$
\end{lemma}


Now we will show that $10 \times 5$ binary $3$-covering arrays are unique.

\begin{theorem}
Every $10 \times 5$ binary $3$-covering array is equivalent to
$$
\small{\left(
\begin{array}{ccccc} 1 & 0 & 0 & 0 & 0\\
1 & 0 & 1 & 1 & 1\\
1 & 1 & 0 & 1 & 1\\
1 & 1 & 1 & 0 & 1\\
1 & 1 & 1 & 1 & 0\\
0 & 1 & 0 & 0 & 0\\
0 & 0 & 1 & 0 & 0\\
0 & 0 & 0 & 1 & 0\\
0 & 0 & 0 & 0 & 1\\
0 & 1 & 1 & 1 & 1 \end{array}\right)}
$$
\end{theorem}
\begin{proof}
It is known in \cite{Slo} that $CAN(3,5,2) = 10$. Let $C$ be a $10 \times 5$ binary $3$-covering array. Since $CAN(2,4,2) = 5$, it follows from Proposition $\ref{Prop:weight}$ that $wt(c^i)=5$ for each $i$. Without loss of generality, we may assume that $c^{1}= (1^50^5)^T$, where $1^50^5$ means $(1,1,1,1,1,0,0,0,0,0)$. Then, $Res(C;c^{1}=1)$ and $Res(C;c^{1}=0)$ are $5 \times 4$ binary $2$-covering arrays. Since $wt(c^i)=5$ for each $i$, by taking complement of columns of $C$ if necessary, we may assume that every column of $Res(C;c^1=1)$ has weight $3$ and every column of $Res(C;c^1=0)$ has weight $2$. The result follows from Lemma \ref{Lem:5by4}.
\end{proof}

Sloane \cite{Slo} constructed a $12 \times 11$ binary $3$-covering array by using Hadamard matrix as follows: Let $H_{12}$ be a normalized Hadamard matrix of order $12$. It is clear that the $12 \times 11$ matrix $C$ which is obtained from $H_{12}$ by deleting the first column of $H_{12}$ and replacing $-1$'s by $0$'s is a binary $3$-covering array. We now prove that this is essentially unique way to obtain a $12 \times 11$ binary $3$-covering array.

Before starting, we introduce three $6 \times 10$ binary $2$-covering arrays and a $6 \times 4$ binary $2$-covering array.

\begin{eqnarray*}\label{Equation:A}
A \ ={\small \left( \begin{array}{cccccccccc}
1 & 1 & 1 & 1 & 1 & 1 & 1 & 1 & 1 & 1 \\
1 & 1 & 1 & 1 & 0 & 0 & 0 & 0 & 0 & 0  \\
1 & 0 & 0 & 0 & 1 & 1 & 1 & 0 & 0 & 0  \\
0 & 1 & 0 & 0 & 1 & 0 & 0 & 1 & 1 & 0  \\
0 & 0 & 1 & 0 & 0 & 1 & 0 & 1 & 0 & 1  \\
0 & 0 & 0 & 1 & 0 & 0 & 1 & 0 & 1 & 1  \end{array} \right)}, \
B_1={\small \left( \begin{array}{cccccccccc}
1 & 1 & 0 & 0 & 0 & 0 & 1 & 1 & 0 & 1  \\
1 & 0 & 1 & 0 & 1 & 0 & 0 & 0 & 1 & 1  \\
1 & 0 & 0 & 1 & 0 & 1 & 0 & 1 & 1 & 0  \\
0 & 1 & 1 & 0 & 0 & 1 & 1 & 0 & 1 & 0  \\
0 & 1 & 0 & 1 & 1 & 1 & 0 & 0 & 0 & 1  \\
0 & 0 & 1 & 1 & 1 & 0 & 1 & 1 & 0 & 0  \end{array} \right)},
\end{eqnarray*}
\begin{eqnarray}
B_2 ={\small \left( \begin{array}{cccccccccc}
1 & 1 & 0 & 0 & 0 & 1 & 0 & 0 & 1 & 1  \\
1 & 0 & 1 & 0 & 0 & 0 & 1 & 1 & 1 & 0  \\
1 & 0 & 0 & 1 & 1 & 0 & 0 & 1 & 0 & 1  \\
0 & 1 & 1 & 0 & 1 & 0 & 1 & 0 & 0 & 1  \\
0 & 1 & 0 & 1 & 0 & 1 & 1 & 1 & 0 & 0  \\
0 & 0 & 1 & 1 & 1 & 1 & 0 & 0 & 1 & 0  \end{array} \right)}, \
D \ ={\small \left( \begin{array}{cccc}
1 & 1 & 0 & 0  \\
1 & 0 & 1 & 0  \\
1 & 0 & 0 & 1  \\
0 & 1 & 1 & 0  \\
0 & 1 & 0 & 1  \\
0 & 0 & 1 & 1  \end{array} \right)}.
\end{eqnarray}

By Theorem \ref{Thm:Sperner} and Corollary \ref{Cor:max2ca1}, we note that the three $6 \times 10$ binary $2$-covering arrays are equivalent.

\begin{theorem}{\label{Thm:12by11}}
There is a unique $12 \times 11$ binary $3$-covering array up
to equivalence.
\end{theorem}


\begin{proof}

Since $CAN(2,10,2) = 6$, $CAN(3,11,2) \geq 12$ by Theorem \ref{Thm:Roux}. Hence $CAN(3,11,2) = 12$. Let $C$ be a $12 \times 11$ binary $3$-covering array. By Proposition \ref{Prop:weight} and Lemma \ref{Lem:dist}, $wt(c^i) = 6$ and $d(c^i, c^j) = 6$ for $1 \leq i \neq j \leq 11$. By the definition of equivalence, we may assume that the first column $c^1$ of $C$ is $c^1 = (1^6 0^6)^T $. Then $Res(C;c^1 = 1)$ and $Res(C;c^1 = 0)$ both are $6 \times 10$ binary $2$-covering arrays. By Corollary \ref{Cor:max2ca1}, we may assume that $Res(C;c^1 = 1)$ is the standard maximal binary $2$-covering array, thus $Res(C;c^1 = 1) = A$, where $A$ is given in Equation (\ref{Equation:A}).
Hence $C$ is of the form;

\begin{eqnarray}\label{Equation:C}
C = {\small \left(
\begin{array}{c|c}

\textbf{1} & Res(C;c^1 = 1) = A        \\
 \hline
\textbf{0} &  \ Res(C;c^1 = 0)

\end{array} \right)},
\end{eqnarray}
where $\textbf{1}$ and $\textbf{0}$ are all $1$'s and all $0$'s column vectors of length $6$, respectively.

Since $wt(c^i) = 6$ and $d(c^i, c^j) = 6$ for $1 \leq i \neq j \leq 11$, the first four columns of $Res(C;c^1 = 0)$ is row equivalent to $D$, which is given in Equation (\ref{Equation:A}).

Using $wt(c^i) = 6$, $d(c^i,c^j)=6$, and the definition of a binary $3$-covering array, it can be easily shown  that
$Res(C;c^1 = 0)$ is row equivalent to $B_1$ or $B_2$, where $B_1$ and $B_2$ are given in Equation (\ref{Equation:A}). Let $C_1$ and $C_2$ be the $3$-covering matrices by putting $Res(C;c^1=0)= B_1$ and $Res(C;c^1=0)= B_2$ in Equation (\ref{Equation:C}), respectively.
Then, it is enough to show that $C_1$ and $C_2$ are equivalent. We can transform $C_1$ into $C_2$ by the following  series of operations:\\
(1) permutation of $8$th row and $9$th row.\\
(2) permutation of $10$th row and $11$th row.\\
(3) permutation of $5$th column and $6$th column.\\
(4) permutation of $5$th row and $6$th row.\\
(5) permutation of $8$th column and $9$th column.\\
(6) permutation of $10$th column and $11$th column.

\end{proof}

By a similar method to the proof in Theorem \ref{Thm:12by11} and using \textrm{Table 1}, we can classify the number of non-equivalent covering arrays satisfying $CAN(3,n,2)=12$ for $6 \leq n \leq 11$:

$$
\begin{array}{cccccccccccccccc}
n & \vline & 6 & 7 & 8 & 9 & 10 & 11\\
\hline
CA(12;3,n,2) & \vline & 9 & 2 & 2 & 1 & 1 & 1
\end{array}
$$
$$
\textrm{Table } 3: \textrm{The number of non-equivalent covering arrays }CA(12;3,n,2).
$$

Colbourn et al. \cite{Col} have already obtained Table 3 by a computer search.

\begin{remark}
We will give a simple proof of Theorem \ref{Thm:12by11} by using the uniqueness of Hadamard matrix of order $12$: Let $C$ be a $12 \times 11$ binary $3$-covering array. By Proposition \ref{Prop:weight} and Lemma \ref{Lem:dist}, $wt(c^i) = 6$ and $d(c^i, c^j) = 6$ for $1 \leq i \neq j \leq 11$. Let $B$ be the $12 \times 12$ matrix obtained from $C$ by adding all $1$ column and replacing $0$'s by $-1$'s. Then $B$ is a Hadamard matrix of order $12$. Hence, Theorem \ref{Thm:12by11} follows from the uniqueness of Hadamard matrix of order $12$.

Although this method is simpler than the proof in Theorem \ref{Thm:12by11} in this case,
we generally use the method in the proof of Theorem \ref{Thm:12by11} when we study the structures of covering arrays.
\end{remark}

\begin{theorem}{\label{Thm:24by12}}
There is a unique $24 \times 12$ binary $4$-covering array up to
unique.
\end{theorem}
\begin{proof}
Since $CAN(3,11,2) = 12$, $CAN(4,12,2) \geq 24$ by Theorem \ref{Thm:Roux}. We will show that $CAN(4,12,2) = 24$ and $24 \times 12$ binary $4$-covering arrays are uniquely determined. Let $C$
be a $24 \times 12$ binary $4$-covering array. By Proposition $\ref{Prop:weight}$ and Theorem \ref{Thm:12by11}, $wt(c^i)=12$ and $d(c^i, c^j) = 12$ for $1 \leq i \neq j \leq 12$.  By the definition of equivalence, we may assume that the first and second column $c^1$ and $c^2$ of $C$ are $c^1 = (1^{12} 0^{12})^T$ and $c^2 = (1^60^61^60^6)^T$.

Since $Res(C;c^1 = 1)$ is a $12 \times 11$ binary $3$-covering array, we may assume that $Res(C;c^1 = 1, c^2 = 1) = A$ and $Res(C;c^1 = 1, c^2 = 0) = B_1$ by Theorem \ref{Thm:12by11} and Equation (\ref{Equation:C}), where $A$ and $B_1$ are given in Equation (\ref{Equation:A}). Since $Res(C;c^2 = 1)$ is also a $12 \times 11$ binary $3$-covering array and  $Res(C;c^1 = 1, c^2 = 1) = A$, it should be either $Res(C:c^1 = 0, c^2 = 1) = B_1$ or $Res(C:c^1 = 0, c^2 = 1) = B_2$. Hence $C$ is of the form;

\begin{eqnarray}\label{Equation:24by12}
C = {\small \left(
\begin{array}{c|c|c}

\textbf{1} & \textbf{1} & Res(C;c^1 = 1, c^2 = 1) = A       \\
 \hline
\textbf{1} & \textbf{0} & Res(C;c^1 = 1, c^2 = 0) = B_1      \\
 \hline
\textbf{0} & \textbf{1} & Res(C;c^1 = 0, c^2 = 1) = B_1 \ \mbox{or} \ B_2       \\
 \hline
\textbf{0} & \textbf{0} & Res(C;c^1 = 0, c^2 = 0)

\end{array} \right)},
\end{eqnarray}
where $\textbf{1}$ and $\textbf{0}$ are all $1$'s and all $0$'s column vectors of length $6$, respectively.

Let $E$ be the first $6 \times 4$ submatrix of $Res(C;c^1 = 0, c^2 =0)$. Since the submatrix $ (c_{ij})_{1 \leq i \leq 24, 3 \leq j \leq 6}$ of $C$ is also a $4$-covering array and $wt(c^i) = 12$ for any column $c^i$ of $C$, the submatrix $E$ is row equivalent to the first $6 \times 4$ submatrix of $\overline{Res(C;c^1 = 1, c^2 = 1)} = \overline{A}$. Hence we may assume that

\begin{eqnarray*}
E \ ={\small \left( \begin{array}{cccc}
0 & 0 & 0 & 0  \\
0 & 0 & 0 & 0  \\
0 & 1 & 1 & 1  \\
1 & 0 & 1 & 1  \\
1 & 1 & 0 & 1  \\
1 & 1 & 1 & 0  \end{array} \right)}.
\end{eqnarray*}

Let $C_1$ be a $4$-covering array with $Res(C;c^1 =0, c^2 = 1) = B_1$ in Equation (\ref{Equation:24by12}). By using the fact that $C$ is a $4$-covering array and $wt(c^7) =12$, the $5$-th column of $Res(C;c^0, c^2 = 0)$ should be $(1,0,1,0,0,1)^T$. Then $d(c^3, c^7) = 14$, which is a contradiction.

Let $C_2$ be a $4$-covering array with $Res(C;c^1 =0, c^2 = 1) = B_2$ in Equation (\ref{Equation:24by12}). By using the fact that $C$ is a $4$-covering array, $wt(c^i) = 12$, and $d(c^i, c^j) = 12$ for $1 \leq i \neq j \leq 12$, it can be shown that $Res(C; c^1 =0, c^2 = 0)$ should be row equivalent to $\overline{Res(C;c^1 = 1, c^2 = 1)} = \overline{A}$.
Thus, $24 \times 12$ binary $4$-covering arrays are uniquely determined.
\end{proof}

\begin{remark}
Colbourn et al. \cite{Col} have also shown that $24 \times 12$ optimal binary $4$-covering arrays are uniquely determined by a computer search.
\end{remark}

We end this section by proving $CAN(5,13,2) \geq 49$.

\begin{theorem}\label{Thm:48by13}
There is no $48 \times 13$ binary $5$-covering array.
\end{theorem}

\begin{proof}
Since $CAN(4,12,2) = 24$, $CAN(5,13,2) \geq 48$ by Theorem \ref{Thm:Roux}. Let $C$ be a $48 \times
13$ binary $5$-covering array. By Proposition \ref{Prop:weight} and Theorem \ref{Thm:24by12},
$wt(c^i)=24$ and $d(c^i, c^j) = 24$ for $1 \leq i \neq j \leq 13$. Since $Res(C;c^1=1)$ is
a $24 \times 12$ binary $4$-covering array, we may assume that
$Res(C;c^1=1,c^2=1,c^3=1) = A$, $Res(C;c^1=1,c^2=1,c^3=0)=B_1$,
$Res(C;c^1=1,c^2=0,c^3=1)=B_2$, and
$Res(C;c^1=1,c^2=0,c^3=0)=\overline{A}$ by Theorem \ref{Thm:24by12}. Since $Res(C;c^2=1)$ is also a $24 \times 12$ binary $4$-covering array, we may assume that
$Res(C;c^1=0,c^2=1,c^3=1)=B_2$ and
$Res(C;c^1=0,c^2=1,c^3=0)=\overline{A}$. Hence $C$ is of the form;

\begin{eqnarray}\label{Equation:24by12}
C = {\small \left(
\begin{array}{c|c|c|c}

\textbf{1} & \textbf{1} & \textbf{1} & Res(C;c^1 = 1, c^2 = 1, c^3 = 1) = A       \\
 \hline
\textbf{1} & \textbf{1} & \textbf{0} & Res(C;c^1 = 1, c^2 = 1, c^3 = 0) = B_1       \\
 \hline
\textbf{1} & \textbf{0} & \textbf{1} & Res(C;c^1 = 1, c^2 = 0, c^3 = 1) = B_2       \\
 \hline
\textbf{1} & \textbf{0} & \textbf{0} & Res(C;c^1 = 1, c^2 = 0, c^3 = 0) = \overline{A}       \\
 \hline
\textbf{0} & \textbf{1} & \textbf{1} & Res(C;c^1 = 0, c^2 = 1, c^3 = 1) = B_2      \\
 \hline
\textbf{0} & \textbf{1} & \textbf{0} & Res(C;c^1 = 0, c^2 = 1, c^3 = 0) = \overline{A}      \\
 \hline
\textbf{0} & \textbf{0} & \textbf{1} &    \\
\hline
\textbf{0} & \textbf{0} & \textbf{0} &
\end{array} \right)},
\end{eqnarray}
where $\textbf{1}$ and $\textbf{0}$ are all $1$'s and all $0$'s column vectors of length $6$, respectively.

By Theorem \ref{Thm:24by12}, $Res(C;c^3=1)$ can not be a $24 \times 12$ binary $4$-covering array. It is a contradiction.
\end{proof}


\begin{thebibliography}{99}
\bibitem{Bie}
J. Bierbrauer and H. Schellwat, Almost independent and weakly biased arrays : efficient constructions and cryptologic applications, Advances in Crptology, CRYPTO 2000, Lecture notes in Computer Science, 2000, 533--543.

\bibitem{Boy}
J. Boyar, G. Brassard, and R. Peralta, Subquadratic zero-knowledge, Journal of the ACM, 42, 1995, 1169--1193.

\bibitem{Bra}
G. Brassard, C. Cr\'{e}peau, and M. Santha, Oblivious transfers and intersecting codes, IEEE Trans. Inform. Theory, 42, 1996, 1769--1780.

\bibitem{Car}
J. Carter and M. Wegman, Universal classes of hash functions, J. Computer and System Sci., 18, 1979, 143--154.


\bibitem{Coh}
D. M. Cohen, S. R. Dalal, M. L. Fredman, and G. C. Patton, The AETG
system: an approach to testing software based on combinatorial
design, IEEE Trans. Software Engineering, 23, 1997, 437--444.

\bibitem{Coh2}
D. M. Cohen, S. R. Dalal, J. Parelius, and G. C. Patton, The
combinatorial design approach to automatic test generation, IEEE
Software, 13, 1996, 83--88.

\bibitem{Coh3}
G. D. Cohen and G. Z\'{e}mor, Intersecting codes and independent families, IEEE Trans. Inform. Theory, 40, 1994, 1872--1881.

\bibitem{Col}
C. J. Colbourn, G. K\'{e}ri, P. P. Rivas Soriano, and J. -C. Schlage-Puchta, Covering and radius-covering arrays: Constructions and classification, Discrete Math. 158, 2010, 1158--1180.

\bibitem{Gal}
A. Gal, A charaterization of span program size and improved lower bounds for monotone span program, In 13th Symposium of the Theory of Computing, 1998, 429--437.

\bibitem{Hil}
A. J. W. Hilton and E. C. Milner, Some intersction theorems for
systems of finite sets, Quart. J. Math. Oxford(2), 18, 1967,
369--384.

\bibitem{Joh}
K. A. Johnson and R. Entringer, Largest induced subgraphs of the
n-cube that contain no 4-cycles, J. Combin. Theory. Series B. 46,
1989, 346--355.

\bibitem{Kap}
S. Kapralov, The nonexistence of the $(21,11,2,2)$ superimposed codes, Fifth International Workshop on Optimal Codes and Related Topics, 2007, 100--104.

\bibitem{Kat}
G. O. H. Katona, Two applications (for search theory and truth
functions) of Sperner type theorems, Periodica Math. Hung. 3, 1973,
19--26.

\bibitem{Kle}
D. J. Kleitman and J. Spencer, Families of k-independent sets,
Discrete Math. 6, 1973, 255--262.

\bibitem{Kim}
H. K. Kim and V. Lebedev, On optimal superimposed codes, J. Combin.
Des. 12, 2004, 79--91.

\bibitem{Lin}
J. H. van Lint and R. M. Wilson, A course in combinatorics,
Cambridge University Press, Cambridge 2001.

\bibitem{Nao}
M. Naor and O. Reingold, On the construction of pseudo-random permutations : Ludy-Rackoff revisited. STOC, 1997, 189--199.

\bibitem{Nur}
K. J. Nurmela, Upper bounds for covering arrays by tabu search,
Discrete Math. 138, 2004, 143--152.

\bibitem{Oh}
D. Y. Oh, A classification of the structrures of some sperner
families and superimposed codes, Discrete Math. 306, 2006,
1722--1731.

\bibitem{Re}
A. R\'{e}nyi, Foundations of Probability, John Wiley and Sons, Inc., New York 1971.



\bibitem{Rou}
G. Roux, k-$propri\acute{e}t\acute{e}s$ dans des tableaux de n
colonnes ; cas particulier de la k-$surjectivit\acute{e}$ et de la
k-$permutativit\acute{e}$, Ph.D.Dissetation, University of Paris 6,
March 1987.

\bibitem{Slo}
N. J. A. Sloane, Covering arrays and intersecting codes, J. Combin.
Des. 1, 1993, 51--63.


\bibitem{Ste}
B. Stevens, L. Moura and E. Mendelsohn, Lower bounds for transversal covers, Des. Codes Cryptogr. 15, 1999, 279--299.

\bibitem{Ste2}
B. Stevens, Transversal covers and packings, PhD thesis, University of Toronto
, 1998.

\bibitem{Ton}
A.-I. Tong, Y.-G. Wu, and L.-D. Li, Room-temperature phosphorimetry
studies of some addictive drugs following dansyl chloride labelling,
Talanta, 43, 1996, 1429--1436.

\end{thebibliography}
\end{document}